\documentclass[twoside,11pt]{amsart}
\usepackage{amsmath,latexsym,amssymb,times,enumerate}

\usepackage{verbatim}

\usepackage{enumitem}

\title{Generators of Residual Intersections}
\author{Yevgeniya Tarasova}
\date{}

\usepackage{parskip}

\DeclareMathOperator{\depth}{depth}

\DeclareMathOperator{\height}{ht}
\DeclareMathOperator{\Fitt}{Fitt}
\DeclareMathOperator{\ann}{ann}
\DeclareMathOperator{\grade}{grade}
\DeclareMathOperator{\Spec}{Spec}
\DeclareMathOperator{\Ext}{Ext}
\DeclareMathOperator{\Kitt}{Kitt}

\newtheorem{thm}{Theorem}[section]
\newtheorem{defn}[thm]{Definition}
\newtheorem{lem}[thm]{Lemma}
\newtheorem{cor}[thm]{Corollary}
\newtheorem{prop}[thm]{Proposition}
\newtheorem{rem}[thm]{Remark}

\begin{document}
\begin{abstract}
    Under suitable technical assumptions, a description is given for the generators of $s$-residual intersections of an ideal $I$ in terms of lower residual intersections, if $s \geq \mu(I)-2$. This implies that $s$-residual intersections can be expressed in terms of links, if $\mu(I) \leq \height(I)+3$ and some other hypotheses are satisfied.
\end{abstract}
\maketitle
\section{Introduction}
In this paper, we give formulas for the generators of residual intersections under suitable assumptions. Let $(R,m)$ be a Noetherian local ring, and let $I$ be a proper $R$-ideal. The ideal $J$ is an $s$-\textit{residual intersection} of $I$ when $J = \mathfrak{a}:I$, where $\mathfrak{a} = (a_1,\dots,a_s) \subsetneq I$ and $\height J \geq s$. The residual intersection is called \textit{geometric} if $\height(I+J) \geq s+1.$ This is a generalization of linkage. Two proper ideals, $I$ and $J$, are said to be \textit{linked} if $I = \mathfrak{a}:J$ and $J = \mathfrak{a}:I$ for some $R$-ideal $\mathfrak{a}$ generated by a length $g$ regular sequence. The ideals $I$ and $J$ are said to be \textit{geometrically linked} if $\height(I+J) = g+1$.

Linkage is a well researched and well understood concept, in particular, when it comes to its connection to Cohen-Macualay and Gorenstein properties. In 1974, Peskine and Szpiro showed that if $R$ is a local Gorenstein ring and $I$ and $J$ are linked with $J = \mathfrak{a}:I$, then $R/J$ is Cohen-Macaulay if and only $R/I$ is Cohen-Macaulay. They also showed that in this case the canonical module of $R/I$ is $J/\mathfrak{a}$ \cite{PS}.

There have been attempts to find similar results for residual intersections. If $R$ is a local Gorenstein ring and $I$ is a proper $R$-ideal such that $R/I$ is Cohen-Macaulay, for an $s$-residual intersection $J$ of $I$, $R/J$ is not necessarily Cohen-Macaulay. However, there is a relationship. The first work exploring this relationship was in a 1983 paper by Huneke where he found that if $I$ is strongly Cohen-Macaulay and satisfies the $G_s$ condition then $R/J$ is Cohen-Macaulay \cite{H}. Papers by Herzog, Vasconcelos and Villarreal \cite{HVV} and by Hunke and Ulrich \cite{HU} explored how to weaken the strongly Cohen-Macaulay assumption. In 1994, Ulrich generalized this work and found settings where one can compute the canonical module of $R/J$ \cite{U}. There is still new research coming out on the computation of the canonical module of $R/J$, on when $R/J$ is Cohen-Macaulay and more recently, the relationship between residual intersections and the Gorenstein properties of Rees algebras \cite{HN, CNT, EU}.

However, the actual computation of residual intersections is not as well understood. Previously, it has been known how to compute an $s$-residual intersection in the following special cases - when, in the notation of the above definition, $I$ is a complete intersection \cite{HU}, when $I$ is a perfect ideal of height two \cite{H}, when $R/I$ is Gorenstein and $I$ is of height three \cite{KU}, and for certain $(\height(I)+1)$-residual intersections \cite{KMU}. Using their more general results, Bouca and Hassanzadeh gave a formula to compute an $s$-residual intersection of $I$ when $I$ is an almost complete intersection \cite{BH}.

In this paper, we study how to compute residual intersections in special cases. Under suitable technical assumptions, we are able to express $s$-residual intersections, for $s \geq \mu(I)-2$, in terms of $(\mu(I)-2)$-residual intersections (Theorem 2.5). Then we show how this implies that $s$-residual intersections can be expressed in terms of links, if $\mu(I) \leq \height(I)+3$ and some other hypotheses are satisfied (Corollary 3.1, Corollary 3.2, Corollary 3.3 and Corollary 3.5).

\section{Main Theorem}

We will be working in settings where the conditions of $G_s$ and weakly $s$-residually $S_2$ are relevant, so we must define them. Let $R$ be a Noetherian ring, $I$ an $R$-ideal, and $s$ an integer. When we say $I$ satisfies  $G_s$, we mean that $\mu(I_p) \leq \dim R_p$ for all $p \in V(I)$ such that $\dim R_p \leq s-1$. An ideal $I$ is said to be \textit{weakly s-residually} $S_2$ if for every $i$ and with $\height (I) \leq i \leq s$ and every geometric $i$-residual intersection $J$ of $I$, $R/J$ is $S_2$.

The condition of being weakly $s$-residually $S_2$ is satisfied by a number of ideals. If $R$ is a local Cohen-Macaulay ring and $I$ is a generically complete intersection ideal, an almost complete intersection ideal and almost Cohen-Macaulay, then by \cite[Theorem 3.3]{HVV}  $I$ satisfies $AN_s^-$ for every $s$, which is a stronger condition than being $s$-residually $S_2$. If $R$ is a Gorenstein local ring and $I$ is an almost almost complete intersection ideal and Cohen-Macaulay, then by \cite[page 259]{AH} $I$ is strongly Cohen-Macaulay and thus by \cite[Theorem 4.5]{CNT} $I$ satisfies $AN_s$ for every $s$, which is a stronger condition than being $s$-residually $S_2$. If $R$ is a Gorenstein local ring and $I$ is licci then by \cite[5.3]{HU} $I$ satisfies $AN_s$ for every $s$. Examples of licci ideals include perfect ideals of grade 2 (\cite{A}, \cite{G}) and perfect Gorenstein ideals of grade 3 \cite{W}. One should note that if $R$ is a local Cohen-Macaulay ring and $I$ is a strongly Cohen-Macaulay ideal then by \cite[Theorem 4.5]{CNT} $I$ satisfies $AN_s$. Another set of examples of $s$-residually $S_2$ ideals are ideals generated by submaximal minors of generic symmetric matrices \cite{K}. The defining ideals of general projections also produce large classes of ideals that are $s$-residually $S_2$ \cite[Section 5]{CEU}.

Before we can begin the proof of our main theorem we must prove a few preliminary lemmas, the first of which is a quick computation for a special case of residual intersections.

\begin{lem}Let $(R,m,k)$ be a Noetherian local ring with $|k| = \infty$. Let $I$ be a proper $R$-ideal and  $\mathfrak{a} \subseteq I$. If $\mu(I/\mathfrak{a}) \leq 1$, then $\Fitt_0(I/\mathfrak{a}) = \mathfrak{a}:I$.\end{lem}
\begin{proof}
Since $\mu(I/\mathfrak{a}) \leq 1$, there exists an $x \in I$ such that $(x)+\mathfrak{a} = I $. Let $\mathfrak{a}:I = (b_1,\dots,b_n)$. Let $A$ be a $1 \times n$ matrix such that $A = [b_1,\dots,b_n]$. Notice that the following is an exact sequence:
$$R^n \xrightarrow{A} R \xrightarrow{x} I/\mathfrak{a} \rightarrow 0. $$

So $\Fitt_0(I/\mathfrak{a}) = I_1(A) = J$.
\end{proof}

We make use of \cite[Lemma 1.3]{U}, and thus will restate it here for the convenience of the reader:

\begin{lem}\cite[Lemma 1.3]{U} Let $(R,m,k)$ be a Noetherian ring with $|k| = \infty$. Let $M$ be a finitely generated $R$-module, consider (not necessarily distinct) prime ideals $p_1,\dots,p_n$ of $R$, and submodules $N_1, \dots, N_m$ of $M$. Then there exists $x \in M$ such that for every $1 \leq i \leq n$ and $1 \leq j \leq m$ , $\mu((M/N_j+(x))_{p_i}) = \max\{0,\mu((M/N_j)_{p_i}) - 1 \}$

\end{lem}

The following lemma, while not critical to the proof of our main theorem, establishes the setting in which the main theorem is relevant.

\begin{lem} Let $(R,m,k)$ be a Noetherian local ring with $|k| = \infty$. Let $I$ be a proper $R$-ideal and  $\mathfrak{a} \subsetneq I$. Suppose $I$ satisfies $G_s$, $\mathfrak{a} : I$  is an $s$-residual intersection and $\mu(\mathfrak{a}+mI/mI) = t \leq s$. Then there exists a generating sequence $a_1,\dots,a_s$  of $\mathfrak{a}$ such that for every subset $\{\nu_1,\dots,\nu_i\} \subseteq \{1,\dots,s\}$, the following is true:

\begin{enumerate}
	\item $\height(a_{\nu_1},\dots,a_{\nu_i}):I \geq i$ for $0 \leq i \leq s$;
	\item $\mu((\mathfrak{a}/(a_{\nu_1},\dots,a_{\nu_i}))_p) \leq \dim R_p - i$ whenever $p \in V(I)$ and $i \leq \dim R_p \leq s-1$;
	\item $\mu((a_1,\dots,a_t) + mI/mI) = t$.
\end{enumerate}
\end{lem}

\begin{proof}
By induction on $0 \leq l \leq s$, we are going to construct elements $a_1,\dots,a_s$ such that for every subset $\{\nu_1,\dots,\nu_i\} \subseteq \{1,\dots,l\}$, the following are true:
\begin{enumerate}[label=(\roman*)]
	\item $\mu(\mathfrak{a}/(a_1,\dots,a_l)) = \max\{0, s-l\};$
	\item $(\mathfrak{a}/(a_{\nu_1},\dots,a_{\nu_i}))_p = 0$ where $p \in \Spec(R)$ with $\dim R_p \leq i-1;$
	\item $\mu((\mathfrak{a}/(a_{\nu_1},\dots,a_{\nu_i}))_p) \leq \dim R_p - i$ whenever $p \in V(I)$ and $i \leq \dim R_p \leq s-1;$
	\item $\mu(\mathfrak{a}/\mathfrak{a}\cap mI + (a_1,\dots,a_l)) = \max\{0,t-l\}.$
\end{enumerate}

When $l = 0$, this is clear as $I$ satisfies $G_s$ and $I_p = a_p$ for all $p \in \Spec(R)$ such that $\dim R_p \leq s-1$.

Let $1 \leq l \leq s$ and assume $a_1,\dots,a_{l-1}$ have already been constructed. To obtain $a_l$, we wish to apply \cite[Lemma 1.3]{U} to the module $M = \mathfrak{a}$ and a finite family, $\mathcal{M}$, of submodules of the form $N = (a_{\nu_1},\dots,a_{\nu_i})$ for every subset $\{\nu_1,\dots,\nu_i\} \subseteq \{1,\dots,l-1\}$ and of the form $N = \mathfrak{a}\cap mI + (a_1,\dots,a_{l-1})$.

For $0 \leq j \leq s$ consider the $j$th Fitting ideals of the $R$-modules $\mathfrak{a}/N$, $F_j = \Fitt_j(\mathfrak{a}/N)$, which define the loci $V(F_j) = \{ p \in \Spec(R) \; | \; \mu((\mathfrak{a}/N)_p) > j\}$. Now, let $\mathcal{Q}$ be the finite subset of $\Spec(R)$ consisting of $m$, the maximal ideal, and of all minimal primes in $V(F_0)$  and all minimal primes ideals in $V(I+F_j)$, for $0\leq j \leq s$ and every $N \in \mathcal{M}$. By \cite[Lemma 1.3]{U}, there exists $a_l \in \mathfrak{a}$ such that $\mu((\mathfrak{a}/N+(a_l))_p) = \max\{0,\mu((\mathfrak{a}/N)_p-1)\}$ for every $p \in \mathcal{Q}$ and ever $N \in \mathcal{M}$.

Note that $a_1, \dots, a_k$ satisfy (i) and (iv) as $m \in \mathcal{Q}$ and, by induction, $\mu(\mathfrak{a}/(a_1,\dots,a_{l-1})) = \max\{0, s-(l-1)\}$ and $\mu\Big(\mathfrak{a}/\big(\mathfrak{a}\cap mI + (a_1,\dots,a_{l-1})\big)\Big) = \max\{0,t-(l-1)\}$. For (ii) and (iii), it suffices to check for factor modules of the form $\mathfrak{a}/(a_{\nu_1},\dots,a_{\nu_i},a_l) = \mathfrak{a}/(N+(a_l))$ for $N \in \mathcal{M}\backslash \{\mathfrak{a}\cap mI + (a_1,\dots,a_{l-1})\}$. Write $F_j = \Fitt_j(\mathfrak{a}/N)$.

To prove (ii), let $p \in \Spec(R)$ with $\dim R_p \leq i$. If $p \notin V(F_0)$, then $(\mathfrak{a}/N)_p = 0$ and we are done. So, assume that $F_0 \subseteq p$. Furthermore, by induction and (ii) applied to $\mathfrak{a}/N$, $\height F_0 \geq i$, which implies that $\dim R_p = i$ and thus $p$ is minimal in $V(F_0)$. Hence, $p \in \mathcal{Q}$. So, by our choice of $a_l$ we conclude that $\mu((\mathfrak{a}/N+(a_l))_p) = \max \{0, \mu((\mathfrak{a}/N)_p)-1\}$. If $p \notin V(\mathfrak{a})$, then $\mu((\mathfrak{a}/N)_p)\leq 1$. If $p \in V(\mathfrak{a})$, then $p \in V(I)$ since for $p \in \Spec(R)$ with $\dim R_p \leq s-1$, $I_p = \mathfrak{a}_p$. Thus, by induction and (iii) applied to $\mathfrak{a}/N$, $\mu((\mathfrak{a}/N)_p) = 0$. In either case, $\mu((\mathfrak{a}/N+(a_l))_p) = 0$ which shows (ii).

To prove (iii), let $p \in V(I)$ with $i+1 \leq \dim R_p \leq s-1$. Write $j = \dim R_p - i - 1$. Notice $0 \leq j \leq s-i - 2$. If $p \notin V(I+F_j)$, then $p \notin V(F_j)$ and hence $\mu((\mathfrak{a}/N)_p) \leq j$ so we are done. So, assume $I+F_j \subseteq p$. We will show $\height(I+F_j) \geq j+i+1$ for $0 \leq j \leq s-i-2$. Let $p'$ be in $ V(I+F_j)$. Since $p' \in V(F_j)$, $\mu((\mathfrak{a}/N)_{p'}) \geq j+1$. Since $p'$ is also in $V(I)$, $\mu((\mathfrak{a}/N)_{p'}) \leq \dim R_{p'}-i$, by induction and (iii) applied to $\mathfrak{a}/N$. Thus, $\dim R_{p'} \geq j+i+1$ and we are done. Therefore $j+i+1 = \dim R_p \geq \height (I+F_j)  \geq j+i+1$, which shows that $p$ is minimal in $V(I+F_j)$ and thus $p \in \mathcal{Q}$. Now our choice of $a_l$ implies that $\mu((\mathfrak{a}/N+(a_l))_p) = \max\{0,\mu(\mathfrak{a}/N)_p-1\}$. By induction and (iii) applied to $\mathfrak{a}/N$, $\mu((\mathfrak{a}/N)_p) \leq \dim R_p - i$. Thus, $\mu((\mathfrak{a}/N+(a_l))_p) \leq \max \{ 0, \dim R_p-i-1 \} = \dim R_p -i -1$.
\end{proof}

Note that condition (1) in Lemma 2.3 is equivalent to $(a_{\nu_1},\dots a_{\nu_i}):I$ being an $i$-residual for every subset $\{\nu_1,\dots,\nu_i\} \subseteq \{1,\dots,s\}$ and $0\leq i \leq s$. Meanwhile condition (2) implies that $(a_{\nu_1},\dots a_{\nu_i}):I$ is also geometric for every subset $\{\nu_1,\dots,\nu_i\} \subseteq \{1,\dots,s\}$ and $0\leq i \leq s-1$.

The following lemma is critical to proving the main theorem. Note that the proof applies the same technique as the previous lemma. This is a common technique, for instance it was used in \cite{U}.

\begin{lem}
Let $I$ and $\mathfrak{a}$ be as in Lemma 2.3 and suppose $a_1,\dots,a_s$ have been selected as in Lemma 2.3 and that $t < s$. Then, we can select an $x \in I$ such that for $\mathfrak{a}' = (a_1,\dots,a_{s-1},x)$, $\mu(\mathfrak{a}'+mI/mI) = \mu((a_1,\dots,a_t,x) + mI/mI) = t+1$, and for any $\{a_{\nu_1},\dots,a_{\nu_i}\} \subseteq \{a_1,\dots,a_{s-1},x\}$, the following hold:

\begin{enumerate}
	\item $\height(a_{\nu_1},\dots,a_{\nu_i}):I \geq i$ for $0 \leq i \leq s$;
	\item $\mu((\mathfrak{a}/(a_{\nu_1},\dots,a_{\nu_i}))_p) \leq \dim R_p - i$ whenever $p \in V(I)$ and $i \leq \dim R_p \leq s-1$.
\end{enumerate}

\end{lem}

\begin{proof}
Note that $a_1,\dots,a_{s-1}$ satisfy the following conditions for every subset $\{\nu_1,\dots,\nu_i\} \subseteq \{1,\dots,s-1\}$:

\begin{enumerate}[label=(\Roman*)]
    \item $\mu(I/ mI + (a_1,\dots,a_{t})) = n-t;$
	\item $(I/(a_{\nu_1},\dots,a_{\nu_i}))_p = 0$ where $p \in \Spec(R)$ with $\dim R_p \leq i-1;$
	\item $\mu((I/(a_{\nu_1},\dots,a_{\nu_i}))_p) \leq \dim R_p - i$ whenever $p \in V(I)$ and $i \leq \dim R_p \leq s-1.$
\end{enumerate}

We construct $x$ by applying \cite[Lemma 1.3]{U} to the module $M = I$ and a finite family, $\mathcal{M}$, of submodules of the form $N = (a_{\nu_1},\dots,a_{\nu_i})$ for every subset $\{\nu_1,\dots,\nu_i\} \subseteq \{1,\dots,s-1\}$ and of the form $N = mI + (a_1,\dots,a_t)$. For $0 \leq j \leq s$ consider the $j$th Fitting ideals of the $R$-modules $I/N$, $F_j = \Fitt_j(I/N)$, which define the loci $V(F_j) = \{ p \in \Spec(R) \; | \; \mu((I/N)_p) > j\}$. Now, let $\mathcal{Q}$ be the finite subset of $\Spec(R)$ consisting of $m$, the maximal ideal, and of all minimal primes in $V(F_0)$  and all minimal primes ideals in $V(I+F_j)$, for $0\leq j \leq s$ and every $N \in \mathcal{M}$. By \cite[Lemma 1.3]{U}, there exists $x \in I$ such that $\mu((I/N+(x))_p) = \max\{0,(I/N)_p-1\}$ for every $p \in \mathcal{Q}$ and ever $N \in \mathcal{M}$.

We show that $a_1,\dots,a_{s-1},x$ satisfy the following conditions for every subset $\{a_{\nu_1},\dots,a_{\nu_i}\} \subseteq \{a_1,\dots,a_{s-1},x\}$:

\begin{enumerate}[label=(\roman*)]
    \item $\mu(I/ mI + (a_1,\dots,a_{t},x)) = n-t-1;$
	\item $(I/(a_{\nu_1},\dots,a_{\nu_i}))_p = 0$ where $p \in \Spec(R)$ with $\dim R_p \leq i-1;$
	\item $\mu((I/(a_{\nu_1},\dots,a_{\nu_i}))_p) \leq \dim R_p - i$ whenever $p \in V(I)$ and $i \leq \dim R_p \leq s-1.$
\end{enumerate}

Item (i) is clear as $m \in \mathcal{Q}$, $mI+(a_1,\dots,a_{t}) \in \mathcal{M}$ and $\max\{0,\mu(I/mI+(a_1,\dots,a_{t}))-1\} = \max \{0,n-t-1\} = n-t-1$. For (ii) and (iii), it suffices to check factor modules of the form $I/(a_{\nu_1},\dots,a_{\nu_{i-1}},x) = I/N+(x)$ for $N \in \mathcal{M}\backslash \{mI + (a_1,\dots,a_{t})\}$. Write $F_j = \Fitt_j(I/N)$.

To prove (ii), let $p \in \Spec(R)$ with $\dim R_p \leq i-1$. If $p \notin V(F_0)$, then $(I/N)_p = 0$ and we are done. So, assume that $p \in V(F_0)$. Furthermore, by (II) applied to $I/N$, $\height F_0 \geq i-1$, which implies that $\dim R_p = i-1$ and thus $p$ is minimal in $V(F_0)$. Hence, $p \in \mathcal{Q}$. So, by our choice of $x$ we conclude that $\mu((I/N+(x))_p) = \max \{0, \mu((I/N)_p)-1\}$. If $p \notin V(I)$, then $\mu((I/N)_p)\leq 1$. If $p \in V(I)$ by (III) applied to $I/N$, $\mu((I/N)_p) = 0$. In either case, $\mu((I/N+(x))_p) = 0$ which shows (ii).

To prove (iii), let $p \in V(I)$ with $i \leq \dim R_p \leq s-1$. Write $j = \dim R_p - i$. Notice $0 \leq j \leq s-i - 1$. If $p \notin V(I+F_j)$, then $p \notin V(F_j)$ and hence $\mu((I/N)_p) \leq j$, and we are done. So, assume $p \in V(I+F_j)$. We will show $\height(I+F_j) \geq j+i$ for $0 \leq j \leq s-i$. Suppose there exists $p' \in V(I+F_j)$ with $\dim R_{p'} \leq j+i-1 \leq s-1$. Since $p' \in V(F_j)$, $\mu((I/N)_{p'}) \geq j+1$. Since $p'$ is also in $V(I)$, $\mu((I/N)_{p'}) \leq \dim R_{p'} -i + 1$ by (III). Thus, $\dim R_{p'} \geq j+i$, which is a contradiction. Therefor $j+i = \dim R_p \geq \height (I+F_j)  \geq j+i$, which shows that $p$ is minimal in $V(I+F_j)$ and thus $p \in \mathcal{Q}$. Now our choice of $x$ implies that $\mu((I/N+(x))_p) = \max\{0,\mu(I/N)_p-1\}$. By (III) applied to $I/N$, $\mu((I/N)_p) \leq \dim R_p - i + 1$. Thus, $\mu((I/N+(x))_p) \leq \max \{ 0, \dim R_p - i \} = \dim R_p -i$.\end{proof}
Now, we are going to prove our main theorem.
\begin{thm}Let $(R,m,k)$ be a local Cohen-Macaulay ring with $|k| = \infty$. Let $I$ be a proper $R$-ideal and  $\mathfrak{a} \subsetneq I$. Let $\mu(I) = n$ and $s \geq n-2$. Suppose $I$ satisfies $G_s$ and is weakly $(s-2)$-residually $S_2$, and $\mathfrak{a} : I$  is an $s$-residual intersection. For any generating sequence $a_1, \dots, a_s$ of $\mathfrak{a}$ as in Lemma 2.3 one has
$$\mathfrak{a}:I = \Fitt_0(I/\mathfrak{a}) + \sum_{\{a_{\nu_1},\dots,a_{\nu_{n-2}}\} \subseteq \{a_1,\dots,a_s\}} (a_{\nu_1},\dots,a_{\nu_{n-2}}):I.$$
\end{thm}

\begin{proof}
We induct on $s$. When $s = n-2$, this is clear.
So, suppose $s \geq n-1$ and that the theorem is true for $s-1$. We use reverse induction on $t = \mu((\mathfrak{a}+mI)/mI)$. Begin with $t = n-1$, which implies that $\mu(I/\mathfrak{a}) = 1$, and thus applying Lemma 2.1, we have that $\mathfrak{a}:I = \Fitt_0(I/\mathfrak{a})$. Finally, noting that $$\sum_{\{a_{\nu_1},\dots,a_{\nu_{n-2}}\} \subseteq \{a_1,\dots,a_s\}} (a_{\nu_1},\dots,a_{\nu_{n-2}}):I \subseteq \mathfrak{a}:I$$

we have, $$\mathfrak{a}:I = \Fitt_0(I/\mathfrak{a}) + \sum_{\{a_{\nu_1},\dots,a_{\nu_{n-2}}\} \subseteq \{a_1,\dots,a_s\}} (a_{\nu_1},\dots,a_{\nu_{n-2}}):I. $$

Now, suppose $t < n-1$ and that the theorem is true for $s$ and $t+1$. By Lemma 2.4, we can construct $\mathfrak{a}'$ generated by $a_1,\dots,a_{s-1},x$ such that, $\mu(\mathfrak{a}' + mI/mI) = t+1$ and the generators satisfy the hypothesis of our theorem. By our choices of $a_1,\dots,a_t,x$, we can select $n-t-1$ general elements $x_{t+1},\dots,x_{n-1}$ such that $I = (x_1,\dots,x_n)$ where $x_i = a_i$ for $i \leq t$ and $x_{n} = x$.

Let $J_{s-1} = (a_1,\dots,a_{s-1}):I$ and $\mathfrak{a}_{s-1} = (a_1,\dots,a_{s-1})$. Let $"\bar{\;}"$ represent images in $\overline{R} = R/J_{s-1}$. Applying \cite[Corollary 3.6.a, Proposition 3.3.b]{CEU}, we get that $\overline{a_s}$ and $\overline{x}$ are non-zerodivisors on $\overline{R}$. In particular, $\grade(\overline{I}) \geq 0$. As $\overline{x}$ and $\overline{a_s}$ are non-zerodivisors, we have that
\begin{equation}
    x(\mathfrak{a}:I)+J_{s-1} = a_s(\mathfrak{a}':I)+J_{s-1}.
\end{equation}
We claim that, similarly,
\begin{equation}
\begin{split}
  a_s\Bigg(\Fitt_0(I/\mathfrak{a}') + \sum_{\{a_{\nu_1},\dots,a_{\nu_{n-3}}\} \subseteq \{a_1,\dots,a_{s-1}\}} (a_{\nu_1},\dots,a_{\nu_{n-3}},x):I \Bigg) + J_{s-1} =  \\
  x \Bigg(\Fitt_0(I/\mathfrak{a}) + \sum_{\{a_{\nu_1},\dots,a_{\nu_{n-3}}\} \subseteq \{a_1,\dots,a_{s-1}\}} (a_{\nu_1},\dots,a_{\nu_{n-3}},a_s):I \Bigg) + J_{s-1}.
\end{split}
\end{equation}

To prove this, first we show that $x(\Fitt_0(I/\mathfrak{a})) + J_{s-1} = a_s(\Fitt_0(I/\mathfrak{a}')) + J_{s-1}$. Letting $a_s = \sum_{i=1}^{n}c_ix_i$, we can define $n \times s$ matrices $B$ and $B'$ where $[a_1,\dots,a_s] = [x_1,\dots,x_{n}]B$ and $[a_1,\dots,a_{s-1},x] = [x_1,\dots,x_{n}]B'$ such that the only difference between them is that the last column of $B$ is $[c_1,\dots,c_{n}]^t$ and the last column of $B'$ is $[0,\dots,0,1]^t$.
Consider an $n \times k$ matrix $A$ such that $R^k \xrightarrow{A} R^{n} \xrightarrow{[x_1,\dots,x_{n}]} I \rightarrow 0$ is exact.

Notice that $$R^{k+s} \xrightarrow{[A|B]} R^{n} \xrightarrow{[x_1,\dots,x_{n}]} I/\mathfrak{a} \rightarrow 0$$ is exact and $$R^{k+s} \xrightarrow{[A|B']} R^{n} \xrightarrow{[x_1,\dots,x_{n}]} I/\mathfrak{a}' \rightarrow 0$$ is exact. As $\Fitt_0(I/\mathfrak{a}) = I_{n}([A|B])$ and $\Fitt_0(I/\mathfrak{a}') = I_{n}([A|B'])$, we can reduce down to looking at $(n) \times (n+1)$ matrices $D$ and $D'$ where the only difference is that the last column of  $D$ is $[c_1,\dots,c_{n}]^t$ and the last column of $D'$ is $[0,\dots,0,1]^t$. Let $E$ be $D$ with the last column removed. Equivalently, $E$ is $D'$ with the last column removed. Let $M_{i,j}$ be the cofactors of $E$. So we have the following,

$$I_{n}(D) = \Bigg(\det(E),\sum_{i=1}^{n}c_iM_{i,1},\dots,\sum_{i=1}^{n}c_iM_{i,n}\Bigg)$$ and

$$I_{n}(D')= (\det(E),M_{n,1},\dots,M_{n,n}).$$

Note that $\overline{[x_1,\dots,x_{n}]E}=0$ which gives that $\det(\overline{E})\:\overline{x_{n}}=0$. Since $\overline{x_{n}}$ is a non-zerodivisor, $\det(\overline{E}) = 0$. Thus $$x\;I_{n}(D) + J_{s-1} =  \Bigg(\sum_{i=1}^{n}c_ix_{n}M_{i,1},\dots,\sum_{i=1}^{n}c_ix_{n}M_{i,n}\Bigg) + J_{s-1}$$ and

$$a_{s}\;I_{n}(D') + J_{s-1}= \Bigg(\sum_{i=1}^{n}c_ix_i M_{n,1},\dots,\sum_{i=1}^{n}c_ix_i M_{n,n}\Bigg) + J_{s-1}.$$

Let $N$ be $\overline{[A|B]}$ (equivalently $\overline{[A|B']}$) with the last column removed. Notice $$\overline{R}^{k+s-1} \xrightarrow{N} \overline{R}^{n} \xrightarrow{[x_1,\dots,x_{n}]} I/\mathfrak{a}_{s-1}\xrightarrow{} 0 $$ is exact. By \cite[Corollary 3.4]{CEU}, $I/\mathfrak{a}_{s-1} \cong \overline{I}$. Since $\grade(\overline{I})>0$, $I_{n}(N) \subseteq \ann_{\overline{R}}I/\mathfrak{a}_{s-1} = \ann_{\overline{R}} \overline{I} = 0$. For for any fixed $j$, $[\overline{M_{1,j}},\dots,\overline{M_{n,j}}]N = 0$, as all entries are determinants of $n \times n$ sub-matrices of $N$ and thus $0$. Note that $[\overline{x_1},\dots,\overline{x_{n}}]\overline{N}=0$. As $\grade(\overline{I}) > 0$, the kernal of the dual of $N$ has rank $1$. Thus, for any fixed $j$, $[\overline{x_1},\dots,\overline{x_{n}}]$ and $[\overline{M_{1,j}},\dots,\overline{M_{n,j}}]$ are linearly dependent. Thus $\overline{x_{n}M_{i,j}}=\overline{x_iM_{n,j}}$, which gives $x(I_{n}(D)) + J_{s-1} =  a_{s}(I_{n}(D')) + J_{s-1}$. It follows that \begin{equation}
    x(\Fitt_0(I/\mathfrak{a})) + J_{s-1} = a_s(\Fitt_0(I/\mathfrak{a}')) + J_{s-1}.
\end{equation}

Now we show a similar relationship holds between the sums of $(n-2)$-residual intersections. First, note that letting $\{\nu_1,\dots,\nu_{n-3}\} \subseteq \{1,\dots,s-1\}$, we have $(a_{\nu_1},\dots,a_{\nu_{n-3}}):I \subseteq J_{s-1}$. Recall $a_s$ and $x$ are nonzero divisors on $R/(a_{\nu_1},\dots,a_{\nu_{n-3}}):I$. Thus, $$a_s((a_{\nu_1},\dots,a_{\nu_{n-3}},x):I)+J_{s-1} = x((a_{\nu_1},\dots,a_{\nu_{n-3}},a_s):I)+J_{s-1}.$$ So, we have

\begin{equation}
\begin{split}
    a_s\Bigg(\sum_{\{a_{\nu_1},\dots,a_{\nu_{n-3}}\} \subseteq \{a_1,\dots,a_{s-1}\}} (a_{\nu_1},\dots,a_{\nu_{n-3}},x):I \Bigg) + J_{s-1} =\\
    x \Bigg(\sum_{\{a_{\nu_1},\dots,a_{\nu_{n-3}}\} \subseteq \{a_1,\dots,a_{s-1}\}} (a_{\nu_1},\dots,a_{\nu_{n-3}},a_s):I \Bigg) + J_{s-1}.
\end{split}
\end{equation}

Thus, equation (2) holds by (3) and (4).

By induction on $s$, we have

$$J_{s-1} = \Fitt_0(I/\mathfrak{a}_{s-1}) + \sum_{\{a_{\nu_1},\dots,a_{\nu_{n-2}}\} \subseteq \{a_1,\dots,a_{s-1}\}} (a_{\nu_1},\dots,a_{\nu_{n-2}}):I.$$

Likewise, by decreasing induction on $t$, we have

$$\mathfrak{a}':I = \Fitt_0(I/\mathfrak{a}') + \sum_{\{a_{\nu_1},\dots,a_{\nu_{n-2}}\} \subseteq \{a_1,\dots,a_{s-1},x\}} (a_{\nu_1},\dots,a_{\nu_{n-2}}):I,$$

so
\begin{equation}
    \mathfrak{a}':I + J_{s-1} = \Fitt_0(I/\mathfrak{a}') + \sum_{\{a_{\nu_1},\dots,a_{\nu_{n-3}}\} \subseteq \{a_1,\dots,a_{s-1}\}} (a_{\nu_1},\dots,a_{\nu_{n-3}},x):I \: + J_{s-1}.
\end{equation}

Putting equations (1), (2) and (5) together, we have:
\begin{eqnarray*}
x(\mathfrak{a}:I)+J_{s-1} &=& a_s(\mathfrak{a}':I)+J_{s-1}\\
&=& a_s\Bigg(\Fitt_0(I/\mathfrak{a}') + \sum_{\{a_{\nu_1},\dots,a_{\nu_{n-3}}\} \subseteq \{a_1,\dots,a_{s-1}\}} (a_{\nu_1},\dots,a_{\nu_{n-3}},x):I \Bigg) + J_{s-1} \\
&=& x \Bigg(\Fitt_0(I/\mathfrak{a}) + \sum_{\{a_{\nu_1},\dots,a_{\nu_{n-3}}\} \subseteq \{a_1,\dots,a_{s-1}\}} (a_{\nu_1},\dots,a_{\nu_{n-3}},a_s):I \Bigg) + J_{s-1}.
\end{eqnarray*}

As $\overline{x}$ is a non-zerodivisor, $$\mathfrak{a}:I  \: + J_{s-1} = \Fitt_0(I/\mathfrak{a})+ \sum_{\{a_{\nu_1},\dots,a_{\nu_{n-3}}\} \subseteq \{a_1,\dots,a_{s-1}\}} (a_{\nu_1},\dots,a_{\nu_{n-3}},a_s):I+J_{s-1}.$$

What we have left to do is "get rid of" $J_{s-1}$ on both sides. First, note $J_{s-1} \subseteq \mathfrak{a}:I$. Moreover, since $I/\mathfrak{a}_{s-1}$ maps onto $I/\mathfrak{a}$, $\Fitt_0(I/\mathfrak{a}_{s-1}) \subset \Fitt_0(I/\mathfrak{a})$. Thus, we have:
\begin{eqnarray*}
\mathfrak{a}:I
&=& \Fitt_0(I/\mathfrak{a}) + \sum_{\{a_{\nu_1},\dots,a_{\nu_{n-3}}\} \subseteq \{a_1,\dots,a_{s-1}\}} (a_{\nu_1},\dots,a_{\nu_{n-3}},a_s):I + \\
&&\sum_{\{a_{\nu_1},\dots,a_{\nu_{n-2}}\} \subseteq \{a_1,\dots,a_{s-1}\}} (a_{\nu_1},\dots,a_{\nu_{n-2}}):I\\
&=& \Fitt_0(I/\mathfrak{a}) + \sum_{\{a_{\nu_1},\dots,a_{\nu_{n-2}}\} \subseteq \{a_1,\dots,a_s\}} (a_{\nu_1},\dots,a_{\nu_{n-2}}):I.
\end{eqnarray*}
\end{proof}

\section{Applications}

We begin our applications by showing that the main theorem provides an alternative proof for some already known results. Our first corollary was previously proven in \cite{HU}.

\begin{cor}
Let $(R,m,k)$ be a local Cohen-Macaulay ring with $|k| = \infty$. Let $I$ be a proper $R$-ideal and  $\mathfrak{a} \subsetneq I$. Suppose $I$ is a complete intersection. Then $$\mathfrak{a}:I = \Fitt_0(I/\mathfrak{a}) + \mathfrak{a}.$$
\end{cor}

\begin{proof}
This follows from Lemma 2.3 and Theorem 2.5.
\end{proof}

The following corollary was also proven, without the generic complete intersection assumption but with a stronger assumption on the depth of $R/I$ in \cite[Corollary 5.5]{BH}.

\begin{cor}
Let $(R,m,k)$ be a local Cohen-Macaulay ring with $|k| = \infty$. Let $I$ be a proper $R$-ideal and  $\mathfrak{a} \subsetneq I$. Suppose $I$ is an almost complete intersection and generically a complete intersection and that $\depth R/I \geq \dim R/I - 1$. Let $\mathfrak{a} : I$  be an $s$-residual intersection. Then $\mathfrak{a}:I = \Fitt_0(I/\mathfrak{a}) + \mathfrak{a}$.
\end{cor}

\begin{proof}
We may assume that $s \geq \height(I)$ and that $\mu(I) = \height(I)+1$. Since $I$ is generically a complete intersection and an almost complete intersection, $I$ is generated by a $d$-sequence. Thus, we can apply \cite[Theorem 3.4.c]{HVV} to get that $I$ satisfies $AN_s^-$. Thus, we can apply  Lemma 2.3 and Theorem 2.5. So, we have $$\mathfrak{a}:I = \Fitt_0(I/\mathfrak{a}) + \sum_{\{a_{\nu_1},\dots,a_{\nu_{g-1}}\} \subseteq \{a_1,\dots,a_s\}} (a_{\nu_1},\dots,a_{\nu_{g-1}}):I.$$
Note, since $g-1 < \height(I)$, $(a_{\nu_1},\dots,a_{\nu_{g-1}}):I = (a_{\nu_1},\dots,a_{\nu_{g-1}})$. Thus, $$\sum_{\{a_{\nu_1},\dots,a_{\nu_{g-1}}\} \subseteq \{a_1,\dots,a_s\}} (a_{\nu_1},\dots,a_{\nu_{g-1}}):I = \mathfrak{a}.$$
\end{proof}

Now, we have the applications of our theorem to previously unknown cases.

\begin{cor} Let $(R,m,k)$ be a local Gorenstein ring with $|k| = \infty$. Let $I$ be a proper $R$-ideal such that $R/I$ is  Cohen-Macaulay and let $\mathfrak{a} \subsetneq I$. Let $\height(I) = g$, $\mu(I) = g+2$, $s \geq g$, and assume $I$ satisfies $G_s$. Let $\mathfrak{a} : I$ be an $s$-residual intersection.  For any generating sequence $a_1, \dots, a_s$ of $\mathfrak{a}$ as in Lemma 2.3 one has $$\mathfrak{a}:I = \Fitt_0(I/\mathfrak{a}) + \sum_{\{a_{\nu_1},\dots,a_{\nu_{g}}\} \subseteq \{a_1,\dots,a_s\}} (a_{\nu_1},\dots,a_{\nu_{g}}):I$$

where each $(a_{\nu_1},\dots,a_{\nu_{g}}):I$ is a link.
\end{cor}

\begin{proof} Note that, as $R$ is Gorenstein and as $\mu(I) = g+2$ and $R/I$ Cohen-Macaulay, by \cite{AH} we have that $I$ is strongly Cohen-Macaulay. Since $I$ is strongly Cohen-Macaulay and satisfies $G_s$, $I$ satisfies $AN_s$ \cite{H}. Thus, we can apply Lemma 2.3 and Theorem 2.5 to get the equality. Further, as $R/I$ is Cohen-Macaulay, $I$ is unmixed. Since $I$ is unmixed and $R$ is Gorenstein, by \cite{PS} we have that each $(a_{\nu_1},\dots,a_{\nu_{g}}):I$ is a link.
\end{proof}

Here, we restate \cite[Corollary 2.18]{KMU} in the way in which we shall use it.

\begin{thm}
Let $R$ a local Gorenstein ring, let $I$ be an ideal of height $g>0$ such that $I$ is generically a complete intersections and assume that $R/I$ is Gorenstein. Let $\mathfrak{a}:I$ be a $(g+1)$-residual intersection of $I$. Then for any generating sequence $a_1,\dots,a_{g+1}$ of $\mathfrak{a}$ such that any length $g$ subsequence is a regular sequence, one has that
$$\mathfrak{a}:I = \sum_{\{a_{\nu_1},\dots,a_{\nu_{g}}\} \subseteq \{a_1,\dots,a_{g+1}\}} (a_{\nu_1},\dots,a_{\nu_{g}}):I$$

if and only if $\,\Ext_R^1(I/I^2,R/I) = 0$.
\end{thm}

Now, we state our final application.

\begin{cor}
Let $(R,m,k)$ be a local Gorenstein ring with $|k| = \infty$. Let $I$ be a proper $R$-ideal such that $R/I$ is Gorenstein and let $\mathfrak{a} \subsetneq I$. Suppose $I$ is of height $g > 0$, $\mu(I) = g+3$, $\Ext^1_{R/I}(I/I^2,R/I) = 0$, and for some $s \geq g+1$, I is $G_{s}$ and weakly $(s-2)$-residually $S_2$. Let $\mathfrak{a}:I$ be an $s$-residual intersection. For any generating sequence $a_1, \dots, a_s$ of $\mathfrak{a}$ as in Lemma 2.3 one has $$\mathfrak{a}:I = \Fitt_0(I/\mathfrak{a}) + \sum_{\{a_{\nu_1},\dots,a_{\nu_{g}}\} \subseteq \{a_1,\dots,a_s\}} (a_{\nu_1},\dots,a_{\nu_{g}}):I$$

where each $(a_{\nu_1},\dots,a_{\nu_{g}}):I$ is a link.
\end{cor}

\begin{proof}
This follows from Lemma 2.3 and Theorem 2.5 and Theorem 3.4.
\end{proof}

 Let $R$ be a Noetherian local ring, $I$ generically a complete intersection ideal and $A = R/I$, then $$\Ext^1_{R/I}(I/I^2,R/I) = T^2(A/R,A).$$ When $T^2(A/R,A)=0$, we say that $R$ is non-obstructed. If $R$  is regular this implies that there are no obstructions for lifting infinitesimal deformations. One should note that by \cite{B} if $I$ is a licci ideal then $T^2(A/R,A) = 0$.

 \section{Another Approach}

Given that our main theorem requires two condition, that the ideal $I$ is both $G_s$ and weakly $s$-residually $S_2$, a natural question arises - can we weaken either condition? Approaching the question of computing the generators of residual intersections using the methods of Bouca and Hassanzadeh \cite{BH}, a partial answer arises. It is possible to eliminate the $G_s$ condition, however, with this method the depth condition is strengthened from $I$ being weakly $s$-residually $S_2$ to $I$ satisfying the $SD_1$ condition. It has yet to be determined if it is possible to eliminate the $G_s$ condition without strengthening the depth condition.

For clarity, we recall the definition of $SD_k$ here.

\begin{defn}
Let $(R,m)$ be a Noetherian local ring of dimension $d$ and $I = (x_1,\dots,x_n)$ be an ideal of grade $g$. Let $k$ be an integer. We say that an ideal satisfies $SD_k$ if

$$\depth(H_i(x_1,\dots,x_n;R)) \geq \min\{d-g, d-n+i+k \}$$

for all $i \geq 0$.
\end{defn}

In the case where $s$ is small, the $G_s$ and weakly $(s-2)$-residually $S_2$ conditions are very weak, especially compared to the $SD_1$ condition. Thus, the methods of Section 3 have an advantage when $s$ is small.

Before we go over how to use the methods of Bouca and Hassanzadeh to derive an analogous result to our main theorem, we will restate a few of the relevant results and definitions from \cite{BH}.

The most relevant of these definitions is the definition of an ideal that Bouca and Hassanzadeh call the disguised residual intersection. While they have an alternate construction for this ideal, we will use an equivalent definition.

\begin{defn}\cite[Theorem 4.9]{BH} \; Let $R$ be a Noetherian ring, $I = (x_1, \dots, x_n)$ and $\mathfrak{a} = (a_1,\dots, a_s) \subseteq I$ be ideals of $R$. Let $B = (c_{ij})$ be an $n\times s$ matrix such that $[a_1,\dots,a_s] = [x_1,\dots,x_n] B$. Let $e_1,\dots,e_n$ be the basis of $K_1(x_1,\dots,x_n; R)$ as an $R$-module. Let $\zeta_j = \sum_{i=1}^n c_{ij}e_i$ and $\Gamma_{\bullet}$ be the R-subalgebra of $K_\bullet(x_1,\dots,x_n; R)$ generated by $\{\zeta_1, \dots, \zeta_s \}$. Let $Z_\bullet = Z_\bullet(x_1,\dots,x_n; R)$ be the R-subalgebra of Koszul cycles. Then the disguised residual intersection is the ideal satisfying $\;\Kitt(\mathfrak{a},I)\cdot e_1\wedge\dots\wedge e_n = \langle \Gamma_\bullet \cdot Z_\bullet \rangle_n$.

\end{defn}

One should note that in \cite[Section 4.2]{BH}, it was shown that the disguised residual intersection does not depend on any choice of generators or the matrix $B$.

\begin{thm} \cite[Theorem 4.23]{BH}
Let $R$ be a Noetherian ring, and keep the same notation as in Definition 4.1. Let $\grade(I) = g$. Let $\tilde{H}_\bullet$ be the R-subalgebra of $K_\bullet(x_1,\dots,x_n; R)$ generated by lifts of the generators of the Koszul homology modules such as to give Koszul cycles. Then

$$\Kitt(\mathfrak{a},I)\cdot e_1\wedge\dots\wedge e_n = \mathfrak{a}\cdot e_1\wedge\dots\wedge e_n + \sum_{i = \max \{0, n-s\}}^{n-g}\Gamma_{n-i} \cdot \tilde{H}_i. $$

\end{thm}

\smallskip

\begin{prop}
\cite[Proposition 4.19]{BH}
Let $R$ be a commutative ring and keep the notation of Definition 4.1. Then, we have

$$\langle \Gamma_\bullet \cdot \langle Z_1(x_1,\dots,x_n;R)\rangle\rangle_n = \Fitt_0(I/\mathfrak{a})\cdot e_1\wedge\dots\wedge e_n.$$
\end{prop}

\smallskip

\begin{rem} \cite[Remark 4.24]{BH} With regards to inclusions, $\Fitt_0(I/\mathfrak{a}) \subseteq \Kitt(\mathfrak{a},I) \subseteq \mathfrak{a}:I$.
\end{rem}

\smallskip

\begin{thm} \cite[Theorem 5.1]{BH}
Let $R$ be a local Cohen-Macaulay ring and $I$ be an ideal of height $g \geq 2$. Assume that $I$ satisfies $SD_1$. Then any $s$-residual intersection $J = \mathfrak{a}:I$ coincides with the disguised residual intersection.
\end{thm}

Now we will derive an analogous result to our main theorem.

\begin{thm}
Let $R$ be a local Cohen-Macaulay ring and $I$ be an ideal of height $g \geq 2$, with $\mu(I) = n$ and which satisfies the $SD_1$ condition. Let $J = \mathfrak{a}:I$ be an $s$-residual intersection with $s \geq n-2$. Then for any generating set $a_1, \dots, a_s$ of $\mathfrak{a}$, one has
$$\mathfrak{a}:I = \Fitt_0(I/\mathfrak{a}) + \sum_{\{a_{\nu_1},\dots,a_{\nu_{n-2}}\} \subseteq \{a_1,\dots,a_s\}} (a_{\nu_1},\dots,a_{\nu_{n-2}}):I.$$
\end{thm}

\begin{proof}Throughout this proof, we use the notation of Theorem 4.2 and Theorem 4.3. Note that this is clear when $s = n-2$. So, assume that $s \geq n-1$. If $s = n-1$, then by Theorem 4.3 and Theorem 4.6, $$\mathfrak{a}:I\cdot e_1\wedge\dots\wedge e_n = \mathfrak{a}\cdot e_1\wedge\dots\wedge e_n+\Gamma_{n-1}\cdot\tilde{H}_1 + \sum_{i=2}^{n-g}\Gamma_{n-i}\cdot\tilde{H}_i.$$

Similarly, if $s \geq n$, then by Theorem 4.3 and Theorem 4.6,
$$\mathfrak{a}:I \cdot e_1\wedge\dots\wedge e_n = \mathfrak{a} \cdot e_1\wedge\dots\wedge e_n +\Gamma_{n}\cdot\tilde{H}_0 +\Gamma_{n-1}\cdot\tilde{H}_1 + \sum_{i=2}^{n-g}\Gamma_{n-i}\cdot\tilde{H}_i.$$

It is clear that $\Gamma_{n}\cdot\tilde{H}_0 \subseteq \Fitt_0(I/\mathfrak{a}) \cdot e_1\wedge\dots\wedge e_n$ and Proposition 4.4 implies that $\Gamma_{n-1}\cdot\tilde{H}_1 \subseteq \Fitt_0(I/\mathfrak{a}) \cdot e_1\wedge\dots\wedge e_n$. As $\Fitt_0(I/\mathfrak{a}) \subseteq \mathfrak{a}:I$, we have that if $s \geq n-1$,
$$\mathfrak{a}:I \cdot e_1\wedge\dots\wedge e_n = \mathfrak{a} \cdot e_1\wedge\dots\wedge e_n + \Fitt_0(I/\mathfrak{a}) \cdot e_1\wedge\dots\wedge e_n + \sum_{i=2}^{n-g}\Gamma_{n-i}\cdot\tilde{H}_i.$$

Applying elementary properties of the exterior algebra and Theorem 4.3, one can see that

$$\mathfrak{a} \cdot e_1\wedge\dots\wedge e_n + \sum_{i=2}^{n-g}\Gamma_{n-i}\cdot\tilde{H}_i = \sum_{\{a_{\nu_1},\dots,a_{\nu_{n-2}}\} \subseteq \{a_1,\dots,a_s\}} \Kitt((a_{\nu_1},\dots,a_{\nu_{n-2}}),I) \cdot e_1\wedge\dots\wedge e_n.$$

Note that, by Remark 4.5,
$$ \Kitt((a_{\nu_1},\dots,a_{\nu_{n-2}}),I) \subseteq  (a_{\nu_1},\dots,a_{\nu_{n-2}}):I \subseteq \mathfrak{a}:I. $$

So, putting it all together, we have that $$\mathfrak{a}:I = \Fitt_0(I/\mathfrak{a}) + \sum_{\{a_{\nu_1},\dots,a_{\nu_{n-2}}\} \subseteq \{a_1,\dots,a_s\}} (a_{\nu_1},\dots,a_{\nu_{n-2}}):I.$$

\end{proof}

One should note that, unlike in our main result, the colon ideals in the Theorem 4.7 are not necessarily $(n-2)$-residual intersections.

Now, we will state corollaries that are analogous to the reults of Section 3. We begin with a corollary analogous to Corollary 3.2. Note that Corollary 3.2 requires that $I$ be generically a complete intersection, but only requires $I$ to be almost Cohen Macaulay rather than Cohen Macaulay.

\begin{cor} \cite[Corollary 5.5]{BH}
Let $R$ be a local Cohen-Macaulay ring and let $I$ be an almost complete intersection ideal which is Cohen-Macaulay. Let $J = \mathfrak{a}:I$ be an $s$-residual intersection. Then $\mathfrak{a}:I = \Fitt_0(I/\mathfrak{a})+\mathfrak{a}.$
\end{cor}

The next result is more general than Corollary 3.3, as it does not require the $G_s$ condition.

\begin{cor}
Let $R$ be a local Gorenstein ring and let $I$ be a proper $R$-ideal such that $R/I$ is Cohen-Macaulay and let $\mathfrak{a} \subsetneq I$. Let $\height(I) = g \geq 2$, $\mu(I) = g+2$ and $s \geq g$. Let $\mathfrak{a}:I$ be an $s$-residual intersection. For any generating sequence $a_1,\dots,a_s$ of $\mathfrak{a}$ one has $$\mathfrak{a}:I = \Fitt_0(I/\mathfrak{a}) + \sum_{\{a_{\nu_1},\dots,a_{\nu_{g}}\} \subseteq \{a_1,\dots,a_s\}} (a_{\nu_1},\dots,a_{\nu_{g}}):I.$$
\end{cor}
\begin{proof}
Note that, as $R$ is Gorenstein and as $\mu(I) = g+2$ and $R/I$ Cohen-Macaulay, by \cite{AH} we have that $I$ is strongly Cohen-Macaulay, and thus satisfies $SD_1$. So, we can apply Theorem 4.7.
\end{proof}

The next result is analogous to Corollary 3.5. Note that in Corollary 3.5 we require that $I$ is $G_s$ and weakly $(s-2)$-residually $S_2$ rather than $I$ being generically a complete intersection and satisfying condition $SD_1$.

\begin{cor}
Let $(R,m,k)$ be a local Gorenstein ring with $|k| = \infty$. Let $I$ be a proper $R$-ideal such that $R/I$ is Gorenstein and let $\mathfrak{a} \subsetneq I$. Suppose $I$ is of height $g \geq 2$, $\mu(I) = g+3$, $\Ext^1_{R/I}(I/I^2,R/I) = 0$, and $I$ is generically a complete intersection and satisfies the condition $SD_1$. Let $\mathfrak{a}:I$ be an $s$-residual intersection. For any generating sequence $a_1, \dots, a_s$ of $\mathfrak{a}$ such that any length $g$ subsequence is a regular sequence, one has $$\mathfrak{a}:I = \Fitt_0(I/\mathfrak{a}) + \sum_{\{a_{\nu_1},\dots,a_{\nu_{g}}\} \subseteq \{a_1,\dots,a_s\}} (a_{\nu_1},\dots,a_{\nu_{g}}):I$$

where each $(a_{\nu_1},\dots,a_{\nu_{g}}):I$ is a link.
\end{cor}

\begin{proof}
This follows from Theorem 4.7 and Theorem 3.4.
\end{proof}

It should be noted that in the setting of both Corollary 4.9 and Corollary 4.10 we can always select generators of $\mathfrak{a}$ such that any length $g$ subsequence is a regular sequence.

Now we shall move on to one final notable comparison between the two methods. The following is a partial proof for Conjecture 5.8 from Bouca and Hassanzadeh \cite{BH}, proving that it is true in the case where $I$ satisfies $G_s$ and is weakly $s$-residually $S_2$. It should be noted that, by applying the strategy of Theorem 4.7, this would rederive the results of Section 3. However, the methods used in Section 3 are more elementary, as they do not require the notion of disguised residual intersections.

\begin{thm}
Let $R$ be a Cohen-Macaulay ring, then $\mathfrak{a}:I = \Kitt(\mathfrak{a},I)$ whenever $a:I$ is an $s$-residual intersection and $I$ satisfies $G_s$ and is weakly $(s-2)$-residually $S_2$.
\end{thm}

\begin{proof}
As $I$ is $G_s$ we can select generators $a_1,\dots,a_s$ of $\mathfrak{a}$ such that for all $i<s$, $(a_1,\dots,a_i):I$ is a geometric $i$-residual intersection \cite[Corollary 1.6]{U} and $I \cap J_i = \mathfrak{a}_i$ \cite[Corollary 3.6]{CEU}. For all $i \leq s$, let $\mathfrak{a}_i = (a_1,\dots,a_i)$ and $J_i = \mathfrak{a}_i:I$. By \cite[Corollary 3.6, Lemma 2.4]{CEU}, $a_i$ is a non-zerodivsor in $R/J_{i-1}$ and, letting $\overline{\cdot}$ represent images in $R/J_{i-1}$, $\overline{J_i} = (\overline{a_i}):\overline{I}$.

Let $\overline{\cdot}$ represent images in $R/J_i$. We will first show, by induction on $i$ for $0 \leq i < s$ that $\Kitt(\mathfrak{a},I)$ surjects onto $\Kitt(\overline{\mathfrak{a}},\overline{I})$ for all $i$. Let $x_1,\dots,x_n$ be a generating set of $I$ and $H_i(x_1,\dots,x_n,R)$ represent the $i$th homology of $K_\bullet(x_1,\dots,x_i;R)$.

Suppose $i = 0$, then $I \cap J_i = 0$  and by \cite[Lemma 1.4]{H} $H_j(x_1,\dots,x_n,R)$ surjects onto $H_j(x_1,\dots,x_n,\overline{R})$ for all $j$. Combining this with Theorem 4.3 gives us that $\Kitt(\mathfrak{a},I)$ surjects onto $\Kitt(\overline{\mathfrak{a}},\overline{I})$.

Now, suppose $i>0$. Let $\cdot '$ represent images in $R/J_{i-1}$. By the inductive hypothesis $\Kitt(\mathfrak{a},I)$ surjects onto $\Kitt(\mathfrak{a}',I')$. Note that $a_i$ is a non-zerodivsor in $R'$ and let $\cdot ''$ be the image in $R/(a_i)+J_{i-1}$. By \cite[Theorem 4.27]{BH}, $\Kitt(\mathfrak{a}',I')$ surjects onto $\Kitt(\mathfrak{a}'',I'')$. Note that $J_i'' = 0:I''$. (Explain why this is the same as case $i = 0$), $\Kitt(\mathfrak{a}'',I'')$ surjects onto $\Kitt(\overline{\mathfrak{a}},\overline{I})$, thus we are done.

Now we prove that $\Kitt(\mathfrak{a},I) = \mathfrak{a}:I$. Let $\height(I)=g$. Note that if $s \leq g$, then by \cite[Proposition 5.10]{BH}, we are done.

So, assume $s \geq g+1$. Let $\overline{\cdot}$ represent images in $R/J_{s-1}$. Note that $\overline{\mathfrak{a}} = (\overline{a_s})$. Thus $\Kitt(\overline{\mathfrak{a}},\overline{I}) = \Kitt((\overline{a_s}),\overline{I})$. Since $a_s$ is a non-zerodivsor on $\overline{R}$, we have that $\grade((\overline{a_s}):\overline{I}) \geq 1$  and thus, by \cite[Remark 5.11]{BH}, $\Kitt((\overline{a_s}),\overline{I}) = (\overline{a_s}):\overline{I}$. Also note that $(\overline{a_s}):\overline{I} = \overline{\mathfrak{a}:I}$. So, putting this all together gives us that $\Kitt(\overline{\mathfrak{a}},\overline{I}) = \overline{\mathfrak{a}:I}$.

Since $\Kitt(\mathfrak{a},I)$ surjects onto $\Kitt(\overline{\mathfrak{a}},\overline{I})$, and by Remark 4.5, $\Kitt(\mathfrak{a},I) \subseteq \mathfrak{a}:I = \mathfrak{a}:I + J_{s-1}$, we conclude that $\Kitt(\mathfrak{a},I) = \mathfrak{a}:I$.

\end{proof}

\end{document}